\RequirePackage{fix-cm}
\documentclass[smallextended,envcountsect,envcountsame]{svjour3}

\usepackage{amsmath,amssymb,mathtools}
\usepackage{microtype}
\usepackage[hidelinks]{hyperref}

\smartqed
\renewenvironment{proof}[1][Proof]
  {\par\addvspace{\medskipamount}\noindent\textit{#1.}\enspace\normalfont}
  {\qed\par\addvspace{\medskipamount}}
\renewcommand{\makeheadbox}{}
\renewcommand{\subclassname}{\textbf{Mathematics Subject Classification (2020)}\enspace}
\journalname{Journal of Optimization Theory and Applications}
\date{}

\newcommand{\R}{\mathrm{I\!R}}
\newcommand{\trans}{\mathsf{T}}
\newcommand{\norm}[1]{\lVert#1\rVert_2}

\title{Convergence of Cyclic Steepest Descent for Quadratics with Two Distinct Eigenvalues}
\titlerunning{Cyclic Steepest Descent with Two Eigenvalues}

\author{Yu Li \and Qihang Wang}
\institute{Yu Li \at
Key Laboratory of Quantum Theory and Applications of MoE, Lanzhou Center for
Theoretical Physics, and Key Laboratory of Theoretical Physics of Gansu Province,
Lanzhou University, Lanzhou, 730000, Gansu, China\\
\email{bosonization@163.com}
\and Qihang Wang, Corresponding author \at
School of Mathematical Sciences, Peking University, Beijing 100871, China\\
\email{2501110049@stu.pku.edu.cn}}

\hypersetup{
  pdftitle={Convergence of Cyclic Steepest Descent for Quadratics with Two Distinct Eigenvalues},
  pdfauthor={Yu Li and Qihang Wang},
  pdfsubject={Exact convergence classification for cyclic steepest descent},
  pdfkeywords={cyclic steepest descent, exact line search, convex quadratic optimization, R-superlinear convergence, exceptional set}
}

\begin{document}

\maketitle
\thispagestyle{plain}

\begin{abstract}
The cyclic steepest descent method computes the exact steepest descent
stepsize at the start of each cycle and reuses it for a fixed number $m\ge2$
of iterations.  We study this method for strictly convex quadratics whose
Hessian has exactly two distinct eigenvalues, of arbitrary multiplicities,
and determine its behavior from every initial point.  At the minimizer no
update is needed.  For nonoptimal initializations, a gradient in one
eigenspace gives one-step termination.  If both eigenspace components are
nonzero and have equal norm, the stepsize never changes and the error contracts
by the same constant factor at every iteration.  In the remaining case, both
components are nonzero and unequal in norm, and the method converges doubly
exponentially and hence $R$-superlinearly.  Thus $R$-superlinear convergence fails exactly on an
explicit cone of initial points of measure zero.  The proof reduces the method
to an exact scalar recurrence for the ratio of the two eigenspace components
at the start of each cycle.  We also determine the leading term of the
logarithm of the error at every iteration within a cycle, which identifies the
exact exponent of the doubly exponential decay and shows that this convergence
is not $Q$-superlinear.
\end{abstract}


\section{Introduction}

For the strictly convex quadratic on $\R^n$
\begin{equation}\label{eq:intro-q}
  q(x)=\frac12 x^{\trans}Ax-c^{\trans}x,\qquad A=A^{\trans}\succ0,
\end{equation}
cyclic steepest descent (CSD) computes an exact steepest-descent stepsize once
per cycle and reuses it for a fixed number $m$ of updates.  We ask how its
convergence depends on the initial gradient when $A$ has two distinct
eigenvalues.  In this case a single spectral ratio determines whether the
orbit terminates, converges linearly, or converges doubly exponentially.

CSD belongs to the delayed-gradient family of Friedlander, Mart\'inez,
Molina, and Raydan~\cite{friedlander1999}.  Dai and Fletcher analyzed its
asymptotic behavior~\cite{daifletcher2005}, and Dai's ICM 2022 survey describes
$R$-superlinear convergence at $m\ge(n+1)/2$ as likely~\cite{daiicm2022};
for integer $m$ this condition is equivalent to $n<2m$, the superlinear side
of the threshold in the simplified model of Dai and Fletcher.
The ICM Conjectures directory records a universal formulation of this
question~\cite{icmconjectures}.  Exact two- and three-dimensional exceptional
orbits disprove that universal formulation~\cite{liwang2026}.  The question
here is different: which initializations are exceptional, and what happens to
every other orbit in the two-spectrum class?  The two papers thus give
complementary answers: exact exceptions invalidate the universal guarantee,
while the classification here determines every exceptional initialization
and the behavior of every other orbit for two-spectrum Hessians.

In two dimensions, Shen, Li, and Dai~\cite[Theorem~1]{shen2025v1} already
derive the logarithmic ratio recurrence and obtain $R$-superlinear
convergence of the cycle-boundary gradient subsequence; their argument
parameterizes the initial gradient by the logarithm of its component ratio
and uses that this logarithm is nonzero.  Their revised preprint
studies the distinct triangle steepest-descent method~\cite{shen2025v2}.
For $m=2$, the alternating steepest-descent method of Dai, Hager,
Schittkowski, and Zhang is precisely CSD~\cite[Section~3.3]{cbb2006}.
A recent preprint of Gu~\cite[Theorem~2.1]{gu2026} presents an
almost-everywhere doubly exponential convergence theorem for Hessians with
$d$ distinct eigenvalues whenever $2m>d$, including repeated eigenvalues.
Its two-spectrum case therefore contains the almost-everywhere convergence
consequence considered here.  Our purpose is the exhaustive classification:
we extend the scalar argument to arbitrary eigenspace multiplicities and to
the full iteration sequence, identify the excluded initializations, at which
this logarithm vanishes, as the exact exceptional cone and determine its linear rate, and give a
direct description of each phase of every successful mixed orbit.  The
phase analysis also shows that the critical per-cycle exponent
$\log(2m-1)$ of Gu's theorem is attained in the two-spectrum case.

Theorem~\ref{thm:classification}, together with its proof in
Section~\ref{sec:classification} and the doubly exponential
estimate~\eqref{eq:superexp}, was first made public on 1~September~2026 in a
Zenodo deposit~\cite{liwangzenodo2026}, three weeks before Gu's preprint
appeared on 22~September~2026; it was obtained independently of that work.
The two results are complementary.  Gu's theorem covers any number of
distinct eigenvalues and excludes a null set of initial points that is not
described explicitly, whereas Theorem~\ref{thm:classification} is restricted
to two distinct eigenvalues and describes every initialization, including the
exceptional cone and its linear rate.  Proposition~\ref{prop:sharp} and
Remark~\ref{rem:gu} are new in the present version and were written after
Gu's preprint appeared.

The proof of Theorem~\ref{thm:classification} uses spectral-block invariance to obtain an exact scalar power map, then separates
its fixed ratio from its expanding logarithmic ratios.  Section~\ref{sec:method}
gives the block reduction; Section~\ref{sec:classification} proves the
classification; and Section~\ref{sec:rates} determines the precise
logarithmic decay and the distinction between $R$- and $Q$-superlinear
convergence.  The explicit counterexample calculations belong to the
companion paper and are not repeated here.

Ordinary exact-line-search steepest descent instead recomputes its stepsize
at every update.  Its probability map and polarization behavior were studied
by Akaike and Yu~\cite{akaike1959,yu2025}.  Reusing a single step changes the
two-block logarithmic ratio multiplier from $-1$ to $1-2m$, which is the
source of the faster decay here.  Alternating and delayed-gradient methods
provide further context~\cite{daiyuan2003,zoumagoules2022}.

\section{Cyclic Steepest Descent and Spectral Blocks}
\label{sec:method}

Let $x_*=A^{-1}c$ and set $e_k=x_k-x_*$ and $g_k=Ae_k$.
All norms are Euclidean.
A null sequence $\{e_k\}$ converges $Q$-superlinearly if
$\norm{e_{k+1}}/\norm{e_k}\to0$, and $R$-superlinearly if
$\norm{e_k}^{1/k}\to0$, that is, if its root-convergence factor
$\limsup_k\norm{e_k}^{1/k}$ is zero~\cite[Chapter~9]{ortega2000};
equivalently, $\norm{e_k}$ is bounded by a $Q$-superlinearly convergent null
sequence, which is the form used in~\cite{shen2025v1}.  It converges
$R$-linearly if this root factor lies in $(0,1)$.
The iteration stops when $g_k=0$.  The exact
steepest-descent stepsize at a cycle boundary is
\begin{equation}\label{eq:alpha}
  \alpha_{m\ell}^{\mathrm{SD}}
  :=\frac{g_{m\ell}^{\trans}g_{m\ell}}
  {g_{m\ell}^{\trans}Ag_{m\ell}}.
\end{equation}
The CSD updates are
\begin{equation}\label{eq:csd}
  x_{m\ell+i+1}=x_{m\ell+i}
  -\alpha_{m\ell}^{\mathrm{SD}}g_{m\ell+i},
  \qquad i=0,\ldots,m-1.
\end{equation}
Equivalently,
\begin{equation}\label{eq:gupdate}
  g_{k+1}=(I-\alpha_k A)g_k,
\end{equation}
where $\alpha_k=\alpha_{m\ell}^{\mathrm{SD}}$ for
$m\ell\le k<m(\ell+1)$.

Assume now that $A$ has exactly two distinct eigenvalues $0<a<b$.  Let $P$
and $Q=I-P$ be the orthogonal projectors onto the corresponding eigenspaces,
so that
\begin{equation}\label{eq:two-spectrum}
  A=aP+bQ.
\end{equation}
At a cycle boundary write
$u=P g_{m\ell}$ and $v=Q g_{m\ell}$.  When $u$ and $v$ are both nonzero,
define
\begin{equation}\label{eq:rdef}
  r_\ell:=\frac{\norm{u}}{\norm{v}},\qquad
  \mu_\ell:=\frac{a\norm{u}^2+b\norm{v}^2}
  {\norm{u}^2+\norm{v}^2}.
\end{equation}
Then $\alpha_{m\ell}^{\mathrm{SD}}=1/\mu_\ell$ and
$a<\mu_\ell<b$.

\begin{lemma}[Block update]
\label{lem:block}
For a two-spectrum Hessian and a cycle boundary with $u,v\ne0$,
\begin{equation}\label{eq:block-update}
  P g_{m(\ell+1)}=\left(\frac{\mu_\ell-a}{\mu_\ell}\right)^m u,
  \qquad
  Q g_{m(\ell+1)}=(-1)^m
  \left(\frac{b-\mu_\ell}{\mu_\ell}\right)^m v.
\end{equation}
\end{lemma}

\begin{proof}[Proof of Lemma~\ref{lem:block}]
During the cycle the same stepsize $1/\mu_\ell$ is used.  On the $a$- and
$b$-eigenspaces the one-step multipliers are respectively
$1-a/\mu_\ell=(\mu_\ell-a)/\mu_\ell$ and
$1-b/\mu_\ell=-(b-\mu_\ell)/\mu_\ell$.  Taking the $m$th powers gives
\eqref{eq:block-update}.  The proof is complete.
\end{proof}

\section{Exact Ratio Dynamics and Classification}
\label{sec:classification}

\begin{theorem}[Two-spectrum classification]
\label{thm:classification}
Let $A=aP+bQ$ with $0<a<b$ and $P,Q$ orthogonal projectors of nonzero rank,
$P+Q=I$, and $PQ=0$.
For CSD with an integer cycle length $m\ge2$, consider any initial error
$e_0$ and let $g_0=Ae_0$.  If $g_0=0$, then $x_0=x_*$ and no step is taken.
Otherwise exactly one of the following holds.

\begin{enumerate}
\item If $Pg_0=0$ or $Qg_0=0$, the method reaches $x_*$ after one update.
\item If both blocks are nonzero and
$\norm{Pg_0}=\norm{Qg_0}$, then every update uses
$\alpha=2/(a+b)$ and
\begin{equation}\label{eq:equal-orbit}
  \norm{e_{k+1}}=\rho\norm{e_k},
  \qquad \rho:=\frac{b-a}{b+a}\in(0,1).
\end{equation}
Thus the orbit is nonterminating and $R$-linear but not $R$-superlinear.
\item If both blocks are nonzero and
$\norm{Pg_0}\ne\norm{Qg_0}$, then the complete CSD sequence converges
$R$-superlinearly to $x_*$.
\end{enumerate}
In particular, the nonterminating non-$R$-superlinear initial gradients are
exactly the cone
\[
  \mathcal E:=\{g\in\R^n:\norm{Pg}=\norm{Qg}>0\}.
\]
This cone has Lebesgue measure zero.
Since $A$ is invertible, the corresponding exceptional set of initial errors
and initial points is Lebesgue-null as well.
\end{theorem}

\begin{proof}[Proof of Theorem~\ref{thm:classification}]
The first assertion follows because an initial gradient in one eigenspace is
annihilated by the exact step $1/a$ or $1/b$.

Assume henceforth that both blocks are nonzero.  Since $a<\mu_\ell<b$,
neither block multiplier vanishes, so both blocks remain nonzero at every
update.  From
\eqref{eq:rdef},
\begin{equation}\label{eq:mu-components}
  \mu_\ell-a=\frac{b-a}{r_\ell^2+1},
  \qquad
  b-\mu_\ell=\frac{(b-a)r_\ell^2}{r_\ell^2+1}.
\end{equation}
Combining Lemma~\ref{lem:block} with \eqref{eq:mu-components} yields the
exact cycle-boundary recurrence
\begin{equation}\label{eq:ratio-recurrence}
  r_{\ell+1}=r_\ell^{1-2m}.
\end{equation}

If $r_0=1$, then $\mu_\ell=(a+b)/2$ for every $\ell$, and each one-step
multiplier has modulus $\rho=(b-a)/(b+a)$.  The two blocks retain equal norm,
and $e_k=A^{-1}g_k$ is multiplied blockwise by the same factors.  Hence
\eqref{eq:equal-orbit} holds, proving the second assertion.

Suppose $r_0\ne1$ and put $\nu=2m-1\ge3$ and
$s_\ell=|\log r_\ell|$.  Equation~\eqref{eq:ratio-recurrence} gives
$s_\ell=\nu^\ell s_0$.  Let $G_\ell:=\norm{g_{m\ell}}$.  A direct use of
\eqref{eq:block-update} gives
\begin{equation}\label{eq:theta}
 \frac{G_{\ell+1}^2}{G_\ell^2}
 =\frac{(b-a)^{2m}(r_\ell^2+r_\ell^{4m})}
 {(a r_\ell^2+b)^{2m}(r_\ell^2+1)}.
\end{equation}
There is a constant $C=C(a,b,m)$ such that
\begin{equation}\label{eq:theta-bound}
  \frac{G_{\ell+1}}{G_\ell}
  \le C\min\{r_\ell,r_\ell^{-1}\}=C e^{-s_\ell}.
\end{equation}
Indeed, for $r\le1$ use $r^2+r^{4m}\le2r^2$ and
$a r^2+b\ge b$; for $r\ge1$ use
$r^2+r^{4m}\le2r^{4m}$ and $a r^2+b\ge a r^2$.
Iterating \eqref{eq:theta-bound},
\begin{equation}\label{eq:superexp}
  G_\ell\le G_0 C^\ell
  \exp\!\left(-s_0\frac{\nu^\ell-1}{\nu-1}\right).
\end{equation}
Thus $G_\ell^{1/\ell}\to0$.  Between cycle boundaries there are only $m-1$
updates and all their multipliers are bounded by a constant depending on
$a,b$; consequently the full gradient sequence has root factor zero as well.
Since $a\norm{e_k}\le\norm{g_k}\le b\norm{e_k}$, the same is true for the
error sequence.  This proves $R$-superlinear convergence in the third case.

Finally, $\mathcal E$ is contained in the zero set of the quadratic form
$g\mapsto\norm{Pg}^2-\norm{Qg}^2=g^{\trans}(P-Q)g$, which is not identically zero
because $P$ and $Q$ have nonzero rank; the zero set of a nonzero polynomial
has Lebesgue measure zero.  The proof is complete.
\end{proof}

\begin{remark}
The recurrence also explains why the cycle length matters.  For $m=1$ the
exponent in \eqref{eq:ratio-recurrence} is $-1$, so the ratio merely
alternates between $r_0$ and $r_0^{-1}$; the two-update multiplier of each
block is then a nonzero constant, and ordinary exact-line-search steepest
descent converges $R$-linearly but not $R$-superlinearly from every mixed
initialization.  The superlinear conclusion above uses $m\ge2$.
\end{remark}

\section{Sharp Fixed-Cycle Asymptotics}
\label{sec:rates}

The estimate in \eqref{eq:superexp} proves convergence but does not specify
the leading error at intermediate updates.  Retaining the bounded radial
terms gives this information without changing the projective reduction.

\begin{proposition}[Phase-dependent logarithmic decay]
\label{prop:sharp}
Under the hypotheses of Theorem~\ref{thm:classification}, suppose both
initial spectral blocks are nonzero and $r_0\ne1$.  Put
$\nu=2m-1$ and $s=|\log r_0|>0$.  For each $i=0,\ldots,m-1$,
\begin{equation}\label{eq:sharp-phase}
 \lim_{\ell\to\infty}
 \frac{-\log\norm{e_{m\ell+i}}}{\nu^\ell}
 =
 \begin{cases}
 s/(\nu-1),&i=0,\\
 \nu s/(\nu-1),&1\le i<m.
 \end{cases}
\end{equation}
In particular,
\begin{equation}\label{eq:sharp-exponent}
 \lim_{k\to\infty}
 \frac{\log(-\log\norm{e_k})}{k}
 =\frac{\log(2m-1)}{m}.
\end{equation}
Nevertheless, the full error sequence is not $Q$-superlinearly convergent.
\end{proposition}

\begin{proof}[Proof of Proposition~\ref{prop:sharp}]
Let $\delta=b-a$, $s_\ell=\nu^\ell s$, and
\[
 D_\ell:=\log\max\{\norm{Pg_{m\ell}},\norm{Qg_{m\ell}}\}.
\]
If $r_\ell>1$, the $a$-block is larger at the cycle boundary.  Write its
norm as $U$ and the other as $Ue^{-s_\ell}$.  The moduli of the one-step
multipliers on the smaller and larger blocks are, respectively,
\[
 t_\ell=\frac{\delta}{a+b e^{-2s_\ell}},
 \qquad t_\ell e^{-2s_\ell}.
\]
After $i\ge1$ updates the ratio of the originally larger to the originally
smaller block is $\exp((1-2i)s_\ell)<1$.  The other block therefore
determines the maximum.  If $r_\ell<1$, the same argument applies with
$a$ and $b$ exchanged.  In both cases, with
\[
 t_\ell=
 \begin{cases}
 \delta/(a+b e^{-2s_\ell}),&r_\ell>1,\\
 \delta/(b+a e^{-2s_\ell}),&r_\ell<1,
 \end{cases}
\]
we obtain the exact identity
\begin{equation}\label{eq:radial-identity}
 \log\max\{\norm{Pg_{m\ell+i}},\norm{Qg_{m\ell+i}}\}
 =D_\ell-s_\ell+i\log t_\ell,
 \qquad 1\le i\le m.
\end{equation}
Moreover,
\[
 \frac{\delta}{a+b}\le t_\ell\le\frac{\delta}{a},
\]
so $\log t_\ell$ is uniformly bounded.  Taking $i=m$ in
\eqref{eq:radial-identity} and summing gives
\begin{equation}\label{eq:radial-sum}
 D_\ell
 =D_0-s\frac{\nu^\ell-1}{\nu-1}
   +m\sum_{j=0}^{\ell-1}\log t_j
 =-\frac{s}{\nu-1}\nu^\ell+O(\ell).
\end{equation}
The logarithms of the full gradient norm and the maximum block norm
differ by at most $\log\sqrt2$.  Since
$a\norm{e_k}\le\norm{g_k}\le b\norm{e_k}$, replacing either gradient
quantity by the error norm also changes its logarithm by a bounded amount.
Equations~\eqref{eq:radial-identity}--\eqref{eq:radial-sum} prove
\eqref{eq:sharp-phase}.  There are finitely many phases, each with a
positive limiting coefficient, so taking one more logarithm yields
\eqref{eq:sharp-exponent}.

Finally, \eqref{eq:ratio-recurrence} implies $r_\ell\to\infty$ on one
parity of the cycle index and $r_\ell\to0$ on the other.  Along the former
cycles, after the first update the $b$-block dominates both the gradient
and the error, while $\mu_\ell\to a$.  The next update thus satisfies
\[
 \frac{\norm{e_{m\ell+2}}}{\norm{e_{m\ell+1}}}
 \longrightarrow \frac{b-a}{a}>0.
\]
The second update exists because $m\ge2$, including when it is the next
cycle boundary.  Hence the successive error ratio cannot tend to zero,
which proves the last assertion.  The proof is complete.
\end{proof}

\begin{remark}
\label{rem:gu}
For two distinct eigenvalues, Gu's theorem~\cite[Theorem~2.1]{gu2026}
bounds the cycle-boundary gradients by $\exp(-Ce^{\kappa\ell})$ for every
$\kappa<\log(2m-1)$, and the full sequence with exponent $\kappa/m$, outside
a null set.  Equation~\eqref{eq:sharp-exponent} shows that the limiting
per-update exponent of every successful two-spectrum orbit equals
$\log(2m-1)/m$ exactly, so the critical value in that theorem is attained
as a limit and cannot be improved in this class.
\end{remark}

\begin{remark}
Measured per update, the exponent in \eqref{eq:sharp-exponent} is largest
among fixed integer cycle lengths at $m=2$.  Indeed,
$t\mapsto\log(2t-1)/t$ is strictly decreasing for $t\ge3$, and
$\log3/2>\log5/3$ since $3^3>5^2$.  This comparison counts gradient
updates only; it does not compare line-search costs or finite-precision
implementations.
\end{remark}

\section{Conclusions}

For two-spectrum Hessians, cyclic steepest descent has an explicit and
complete initial-condition classification.  Apart from one-step terminating
starts, the equal-norm spectral cone is exactly the set of linearly convergent
exceptions; every other orbit converges doubly exponentially.  The scalar
reduction identifies this null set without discarding any part of the actual
iteration.  The phase-dependent asymptotics also explain why rapid
$R$-superlinear convergence can coexist with nonvanishing ratios between
successive errors.  These conclusions distinguish the failure of a universal
convergence assertion from the almost-everywhere behavior of the same method.

\begin{acknowledgements}

We acknowledge Gewu Intelligence Lab for providing the collaborative research
environment in which this project was developed.
\end{acknowledgements}

\section*{Declarations}

\noindent\textbf{Funding.} No funding was received for this work.

\noindent\textbf{Competing interests.} The authors declare no competing interests.

\noindent\textbf{Data availability.} No datasets were generated or analysed.  The results are established
analytically in the article.

\section*{Disclosure of automated assistance}

The main results of this paper were obtained through a generative-AI
workflow using OpenAI GPT-5.6 Sol, Anthropic Claude Fable 5, OpenAI GPT-6
Astra, OpenAI GPT-6.1 Sol, Anthropic Claude Fable 5.1, and Anthropic Claude
Opus 5.5.
The workflow contributed to the two-spectrum reduction and phase-dependent
asymptotic analysis, algebraic and proof checks, literature review, and
preparation and revision of the LaTeX manuscript.

\end{document}